\documentclass[a4paper,12pt,leqno]{article}
\usepackage[utf8]{inputenc}

\usepackage[top=3cm, left=2.5cm, right=2.5cm, bottom=3cm]{geometry}

\usepackage{hyperref}
\hypersetup{
    colorlinks=true,
    linkcolor=dodger,
    filecolor=dodger,
    urlcolor=dodger,
    citecolor=dodger,
}

\usepackage{amsmath,mathtools}
\usepackage{amsthm, pb-diagram}
\usepackage{amssymb,comment}
\usepackage{amsfonts,graphicx,color}
\usepackage{enumitem, fancyhdr, dsfont}
\usepackage[normalem]{ulem}
\usepackage{fontawesome}
\usepackage{thmtools,cleveref}
\usepackage{stmaryrd}

\usepackage{lineno,multicol}

\usepackage{tikz,float}

\setlist{topsep=0ex}

\makeatletter
\def\namedlabel#1#2{\begingroup
    #2%
    \def\@currentlabel{#2}%
    \phantomsection\label{#1}\endgroup
}
\makeatother

\definecolor{sub0}{RGB}{29,32,137}
\definecolor{sub1}{RGB}{1,71,157}
\definecolor{sub2}{RGB}{1,104,183}
\definecolor{sub3}{RGB}{0,160,234}
\definecolor{sug}{RGB}{0,154,68}
\definecolor{suy}{RGB}{208,219,1}

    \newcommand{\Dwf}{\mathcal{D}}
    \newcommand{\Ewf}{\mathcal{E}}

    \newcommand{\Mwf}{\mathcal{M}}
    \newcommand{\Nwf}{\mathcal{N}}

    \newcommand{\bfrak}{\mathfrak{b}}
    \newcommand{\cfrak}{\mathfrak{c}}
    \newcommand{\dfrak}{\mathfrak{d}}

    \newcommand{\Pbf}{\mathbf{P}}

    \newcommand{\Bor}{\mathbb{B}}

    \newcommand{\Ebb}{\mathbb{E}}

    \newcommand{\Por}{\mathbb{P}}

    \newcommand{\menos}{\smallsetminus}
    
    \DeclareMathOperator{\pts}{\mathcal{P}}

    \newcommand{\Q}{\mathbb{Q}}
    \newcommand{\R}{\mathbb{R}}

    \newcommand{\la}{\langle}
    \newcommand{\ra}{\rangle}

    \DeclareMathOperator{\add}{\mathrm{add}}
    \DeclareMathOperator{\non}{\mbox{\rm non}}
    \DeclareMathOperator{\cov}{\mbox{\rm cov}}
    \DeclareMathOperator{\cof}{\mbox{\rm cof}}

    \DeclareMathOperator{\Lev}{Lv}

    \newcommand{\seq}[2]{\la #1 \colon  #2\ra}
    \newcommand{\set}[2]{\{#1 \colon  #2\}}
    
    \newcommand{\lset}[2]{\left\{#1 \colon  #2\right\}}

    \newcommand{\id}{\mathrm{id}}

    \DeclareMathOperator{\stem}{\mathrm{st}}

    \DeclareMathOperator{\suc}{succ}
    \newcommand{\calY}{\mathcal{Y}}
    
\newcommand{\tbf}{\mathbf{t}}

    \definecolor{carrotorange}{rgb}{0.93, 0.57, 0.13}
    \definecolor{dodger}{rgb}{0.0,0.5,1.0}
    
    \newcommand{\varp}{\varepsilon}

    \newcommand{\Mcal}{\mathcal{M}}
    \newcommand{\Ncal}{\mathcal{N}}

    \DeclareMathOperator{\loss}{loss}

    \newcommand{\Etld}{\tilde{\Ebb}}

\title{Anatomy of $\tilde{\mathbb{E}}$}
\author{Diego A.~Mejía%
        \thanks{Email: \href{mailto:diego.mejia@shizuoka.ac.jp}{\texttt{diego.mejia@shizuoka.ac.jp}}
}}

\date{{\normalsize
Faculty of Science, Shizuoka University\\ Ohya 836, Suruga-ku, Shizuoka 422--8529, Japan}
}

\begin{document}

\makeatletter
\def\@roman#1{\romannumeral #1}
\makeatother

\newcounter{enuAlph}
\renewcommand{\theenuAlph}{\Alph{enuAlph}}


\theoremstyle{plain}
  \newtheorem{theorem}{Theorem}[section]
  \newtheorem{corollary}[theorem]{Corollary}
  \newtheorem{lemma}[theorem]{Lemma}
  \newtheorem{mainlemma}[theorem]{Main Lemma}
  \newtheorem{prop}[theorem]{Proposition}
  \newtheorem{clm}[theorem]{Claim}
  \newtheorem{fct}[theorem]{Fact}
  \newtheorem{fact}[theorem]{Fact}
  \newtheorem{question}[theorem]{Question}
  \newtheorem{problem}[theorem]{Problem}
  \newtheorem{conjecture}[theorem]{Conjecture}
  \newtheorem*{thm}{Theorem}
  \newtheorem{teorema}[enuAlph]{Theorem}
  \newtheorem*{corolario}{Corollary}
  \newtheorem*{scnmsc}{(SCNMSC)}
\theoremstyle{definition}
  \newtheorem{definition}[theorem]{Definition}
  \newtheorem{example}[theorem]{Example}
  \newtheorem{remark}[theorem]{Remark}
  \newtheorem{notation}[theorem]{Notation}
  \newtheorem{context}[theorem]{Context}
  \newtheorem{exer}[theorem]{Exercise}
  \newtheorem{exerstar}[theorem]{Exercise*}
  \newtheorem{assumption}[theorem]{Assumption}

  \newtheorem*{defi}{Definition}
  \newtheorem*{acknowledgements}{Acknowledgements}
  
\numberwithin{equation}{theorem}

\def\sectionautorefname{Section}
\def\subsectionautorefname{Subsection}

\maketitle

\begin{abstract}
We present a detailed general framework to describe the forcing $\tilde{\mathbb{E}}$, defined by Kellner, Shelah and Tan\u asie to prove the consistency with ZFC of an alternative order of Cicho\'n's maximum. Our presentation is close to the framework of tree-creature forcing notions from Horowitz and Shelah. We show that the posets in this class have strong FAM limits for intervals (in recent terminology, they are $\sigma$-FAM-linked) and, furthermore, that they also have strong ultrafilter limits for intervals.
\end{abstract}

\section{Introduction}

\begin{figure}[ht]
\begin{center}
  \includegraphics[scale=1.3]{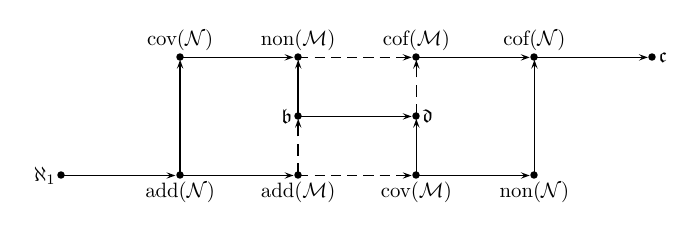}
  \caption{Cicho\'n's diagram. The arrows mean $\leq$ and dotted arrows represent
  $\add(\Mwf)=\min\{\bfrak,\cov(\Mwf)\}$ and $\cof(\Mwf)=\max\{\dfrak,\non(\Mwf)\}$, which we call the \emph{dependent cardinals}.}
  \label{FigCichon}
\end{center}
\end{figure}

\emph{Cicho\'n's maximum} refers to the situation when all non-dependent cardinals in Cicho\'n's diagram (see \autoref{FigCichon}) are pairwise different. 
The first proof of the consistency of Cicho\'n's maximum with ZFC is due to Goldstern, Kellner and Shelah~\cite{GKS}. They first refined a ccc poset from~\cite{GMS}, constructed via a FS (finite support) iteration, to force the non-dependent cardinals on the left side of the diagram pairwise different (i.e.\ with $\non(\Mwf)<\cov(\Mwf)=\cfrak$), and applied \emph{Boolean ultrapowers} to this poset to force, in addition, that the non-dependent cardinals on the right side can also be separated. To force $\bfrak<\non(\Mwf)<\cov(\Mwf)$ with the first ccc poset, we introduced in~\cite{GMS} the notion of \emph{ultrafilter limits for posets} to show that some restrictions of the forcing $\Ebb$, the standard ccc poset that adds an eventually different real, do not add dominating reals along the iteration.

Although there are many possible instances of Cicho\'n's maximum, only four are possible to be forced using FS iterations of ccc posets. The reason is that such iterations add Cohen reals at limits steps, which make them force $\non(\Mwf)\leq \cov(\Mwf)$. Under this restriction, Cicho\'n's diagram gets reduced as illustrated in \autoref{FScichon}, which can only be extended to four different linear orders. The constellation proved consistent in \cite{GMS,GKS} satisfies $\cov(\Nwf)<\bfrak$ and $\dfrak<\non(\Nwf)$.

\begin{figure}[ht]
\centering
\begin{tikzpicture}[scale=0.9, transform shape]
\small{
 \node (aleph1) at (-2,2.5) {$\aleph_1$};
 \node (addn) at (0,2.5){$\add(\Nwf)$};
 \node (covn) at (2,5){$\cov(\Nwf)$};
 \node (b) at (2,0) {$\bfrak$};
 \node (nonm) at (4,2.5) {$\non(\Mcal)$} ;
 \node (d) at (8,5) {$\dfrak$};
 \node (covm) at (6,2.5) {$\cov(\Mcal)$} ;
 \node (nonn) at (8,0) {$\non(\Ncal)$} ;
 \node (cfn) at (10,2.5) {$\cof(\Ncal)$} ;
  \node (c) at (12,2.5) {$\cfrak$};

\foreach \from/\to in {
aleph1/addn, addn/covn, addn/b, covn/nonm, b/nonm, covm/nonn, covm/d, nonn/cfn, d/cfn, cfn/c}
{
\path[-,draw=white,line width=3pt] (\from) edge (\to);
\path[->,] (\from) edge (\to);
}
\path[->,draw=red,line width=1pt] (nonm) edge (covm);

%
%
%

}
\end{tikzpicture}
\caption{Cicho\'n's diagram after a FS iteration of (non-trivial) ccc posets with length of uncountable cofinality.}\label{FScichon}
\end{figure}
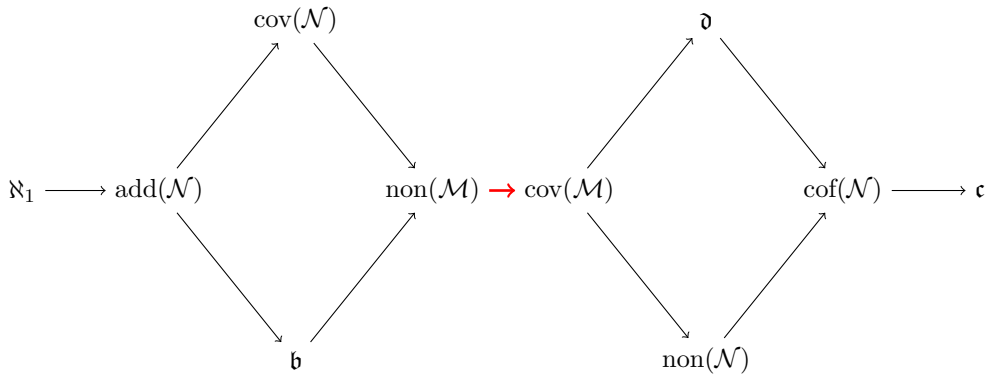

The other possible constellation for the left side of Cicho\'n's diagram in the context of FS iterations is the one where $\bfrak<\cov(\Nwf)$. The challenge to force this constellation is to avoid adding dominating reals when forcing with restrictions of random forcing and $\Ebb$ (to increase $\cov(\Nwf)$ and $\non(\Mwf)$, respectively). From~\cite{GMS} we know that we can use the method of ultrafilter limits for $\Ebb$, but it is unknown (and very unlikely) whether random forcing has ultrafilter limits. On the other hand, the method of fam-limits (fam stands for \emph{finitely additive measure}) from Shelah~\cite{ShCov} (which preceeds ultrafilter limits) works for random forcing, but not for $\Ebb$ (see \autoref{exm:famlk}~\ref{famlkE}). So, one way to solve this issue is to either find a forcing with fam-limits increasing $\non(\Mwf)$, or a forcing with ultrafilter limits adding random reals\footnote{Currently, forcings with ultrafilter-limits can only be iterated when they are $\mu$-centered for small $\mu$, but there is no way a forcing adding random reals can have this property~\cite{Br}.}

Kellner, Shelah and Tan\u asie~\cite{KST} found a forcing with fam-limits that adds eventually different reals (i.e.\ increasing $\non(\Mwf)$), which they denoted by $\Etld$ (\cite[Def.~1.14]{KST}). This is a modification of a tree-creature ccc forcing introduced by Horowitz and Shelah~\cite{HSh}. In this way, they succeeded in forcing, with a ccc FS iteration, the separation of the cardinals of the left side of Cicho\'n's diagram with $\bfrak<\cov(\Nwf)$, which can be used to force Cicho\'n's maximum with $\bfrak<\cov(\Nwf)$ and $\non(\Nwf)<\dfrak$ after Boolean ultrapowers. We remark that the method of Boolean ultrapowers uses large cardinals. On the other hand, we found another method to force the constellations of Cicho\'n's maximum from~\cite{GKS,KST} without using large cardinals~\cite{GKMS}. The consistency of the two remaining constellations of Cicho\'n's maximum compatible with \autoref{FScichon} are still unknown.

In this way, the forcing $\Etld$ plays an important role in forcing different values in Cicho\'n's diagram (and also simultaneously with other cardinal characteristics). However, $\Etld$ is a tree-forcing defined by declaring many parameters (\cite[Def.~1.12]{KST}) and using creature-type norms on the successors of the nodes, which makes it a little bit difficult to digest. Considering this, the proof that $\Etld$ has \emph{strong FAM limits for intervals} (in the terminology from~\cite{KST}) and some other properties become quite involved.

In this note, we make an effort to describe the elements of $\Etld$ in more detail, to give a clearer picture of how this forcing behaves. To achieve this, we do not rely on the many parameters fixed in~\cite{KST}, but we reduce them to a more general framework, closer to~\cite{HSh}, containing much less information. Since this framework depends on a \emph{tree with creatures}, we start by developing in detail the type of creatures we use to define the forcing. So our framework results in a class of forcings of the form $\Etld_\tbf$ where $\tbf$ is what we call a \emph{tree-creature frame}.

Afterwards, we focus on two main results about $\Etld_\tbf$: we show that, when $\tbf$ satisfies certain hypothesis (\autoref{hyplog}), the forcing is $\sigma$-$\calY_*$-linked and $\sigma$-$\Dwf_*$-linked. The first property corresponds to a generalization of strong FAM limits for intervals introduced in~\cite{Uribethesis,CMU}, which can be used in FS iterations to produce large posets with fam-limits. The second property is the version of ultrafilter limits for intervals. The proof of these properties is based on the original proof of~\cite[Lem.~1.20]{KST}.

\section{Tree-creature frames and the forcing}\label{sec:tree}

This section is based on~\cite{HSh,KST}. 
We first look at the atomic structures that we use to define our forcing.

\begin{definition}\label{def:norm}
Let $C$ be a set. 
A \emph{norm of subsets of $C$} is a function $\|\cdot\| \colon \pts(C) \to [0,\infty]$ satisfying:
\begin{enumerate}[label= \normalfont (\roman*)]
\item $\|A\|\leq \|B\|$ whenever $A\subseteq B \subseteq C$.
\item $\|\emptyset\| = 0$.
\end{enumerate}
\end{definition}

\begin{example}[{cf.~\cite[Def.~1.12]{KST}}]\label{ex:norm}
  Given a non-empty set $C$ and a natural number $m>1$, the following is a norm of subsets of $C$:
  \[\|A\|^C_{m}:=\log_m\left(\frac{|C|}{|C|-|A|}\right).\]
 Equivalently, $\|A\|^C_m$ is the solution $x\in\R$ to the equation
  \begin{equation}\label{normchar} 
  |A| = |C|\left(1- m^{-x}\right).
  \end{equation}
  Notice that $\|C\|^C_m = \infty$ and, whenever $|A|=|C|-1$, $\|A\|^C_m = \log_m |C|$.
  
  Typically, $m\ll |C|$. For example, when $|C|>\frac{m}{m-1}$, we obtain that all singletons have norm ${<}1$. Moreover, a set $A$ with norm ${\geq}1$ must have size close to $C$ as long as $m$ is very large. 
\end{example}

\begin{lemma}[{cf.~\cite[Lem.~1.16]{KST}}]\label{normCmprop}
 Let $C$ be a non-empty set and $m>1$ a natural number. Then the norm $\|\cdot\|:=\|\cdot\|^C_m$ satisfies the following properties.
\begin{enumerate}[label=\normalfont(\alph*)]
\item\label{norm1} For $A\subseteq C$, $\|A\| \geq 1$ iff $\displaystyle |A| \geq |C|\left(1-\frac{1}{m}\right)$.

\item\label{norm2} Let $I$ be a finite set, $\bar A=\seq{A_i}{i\in I}$ a sequence of subsets of $C$, and $g\colon I\to [0,1]$ such that $\sum_{i\in I}g(i) =1$. For a real $\varp >0$, consider the set
\[B_{\bar A,g}(\varp)=\lset{k\in C}{\sum\set{g(i)}{i\in I,\ k\in A_i}> 1-\varp}.\]
Then $\|B_{\bar A,g}(\varp)\|\geq \min_{i\in I}\|A_i\| - \log_m \frac{1}{\varp}$.

\end{enumerate}
\end{lemma}
\begin{proof}
\ref{norm1}: By \autoref{normchar}, $\|A\|\geq 1$ iff $|A|=|C|(1-m^{-\|A\|}) \geq |C|(1-m^{-1})$.

\ref{norm2} Let $x:=\min_{i\in I}\|A_i\|$ and $B:=B_{\bar A,g}(\varp)$. Then
\begin{align*}
|C|(1-m^{-x}) & = \sum_{i\in I}g(i)|C|(1-m^{-x}) \leq \sum_{i\in I}g(i)|A_i| = \sum_{i\in I}\sum_{k\in A_i}g(i) \\ 
 & = \sum_{k\in C}\sum\set{g(i)}{i\in I,\ k\in A_i}\\ 
 & = \sum_{k\in B}\sum\set{g(i)}{i\in I,\ k\in A_i} + \sum_{k\in C\menos B}\sum\set{g(i)}{i\in I,\ k\in A_i}\\ 
 & \leq |B| + (|C|-|B|)(1-\varp) =  |B|\varp + |C|(1-\varp).
\end{align*}
Then $|C|(\varp-m^{-x}) \leq |B|\varp$, so
\[|B|\geq |C|\left(1-\frac{m^{-x}}{\varp}\right) = |C|\left(1-m^{-\left(x-\log_m \frac{1}{\varp}\right)}\right).\]
Therefore, $\displaystyle\|B\| \geq x-\log_m \frac{1}{\varp}$.
%
\end{proof}

Property~\ref{norm2} in \autoref{normCmprop} is essential for the forcing arguments. 
We isolate it below as a notion of \emph{co-bigness} as in~\cite{HSh}.

\begin{definition}\label{def:mbig}
   A norm $\|\cdot\|$ of subsets of $C$ is \emph{$m$-co-big} if it satisfies~\ref{norm2} of \autoref{normCmprop}.
\end{definition}

Co-bigness is a property that allows ``homogenizing'' many sets without losing much of the norm. A way to homogenize is taking intersection, for which we may not lose much norm, either.

\begin{lemma}[{cf.~\cite[Lem.~1.16~(c)]{KST}}]\label{lem:mbig}
   If $\|\cdot\|$ is an $m$-co-big norm of subsets of $C$ then, whenever $I$ is a finite set and $\seq{A_i}{i\in I}$ is a sequence of subsets of $C$,
\[\left\|\bigcap_{i\in I}A_i\right\| \geq \min_{i\in I}\|A_i\|-\log_m |I|.\]
\end{lemma}
\begin{proof}
   Apply the $m$-co-bigness to $g(i)=\varp=\frac{1}{|I|}$. In this case, $B_{\bar A,g}=\bigcap_{i\in I}A_i$.
\end{proof}

Our forcing is composed by trees. 
When $T$ is a tree and $t\in T$, we denote by $\suc_T(t)$ the \emph{set of immediate successors of $t$ in $T$.} When $T$ has a single root (i.e.\ only one element of height $0$), the \emph{stem of $T$} is the node $s$ of smallest height such that $|\suc_T(s)|\geq 2$ (if it exists). For any ordinal $\alpha$, $\Lev_\alpha(T)$ denotes the set of nodes of $T$ at level $\alpha$.

\begin{definition}\label{def:treeframe}
  A tuple $\tbf = \la T_*,b_*,\seq{\|\cdot\|_t}{t\in T_*}\ra$ is a \emph{tree-creature frame} when:
  \begin{enumerate}[label=\normalfont(T\arabic*)]
\item $T_*$ is a finitely-branching tree of height $\omega$ with a single root and without maximal nodes. Without loss of generality, we can assume that $T_*$ is a subtree of $\omega^{<\omega}$.

\item\label{tree2} $b_*\colon T_*\to \omega$ such that $\displaystyle \lim_{n\to \infty}\min_{t\in \Lev_n(T_*)} b_*(t) = \infty$.

\item For each $t\in T_*$, $\|\cdot\|_t$ is a $b_*(t)$-co-big norm of subsets of $\suc_{T_*}(t)$ such that all singletons have norm ${<}1$ and $\|\suc_{T_*}(t)\|_t=\infty$.
\end{enumerate} 
\end{definition}

Tree-creature frames are easy to construct. In the practice, $T_*$ and $b_*$ are constructed simultaneously by induction on the height $n$ such that $b^*(t)$ for $t\in \Lev_n(T_*)$ is much larger than $|\Lev_n(T_*)|$ and the $b^*(s)$ defined so far, and $|\suc_{T_*}(t)|\gg b_*(t)$. Co-big norms are easily obtained from~\autoref{ex:norm}. In~\cite[Def.~1.12]{KST}, $b^*(t)$ and $\suc_{T_*}(t)$ only depend on the height of $t$. On the other hand,~\cite{HSh} proceeds by induction on $t$ using the lexicographic order $\triangleleft$ of $\omega^{<\omega}$ satisfying $|s|<|t| \Rightarrow s \mathrel{\triangleleft} t $, and defines $b_*(t)$ much larger than everything defined for $s \mathrel{\triangleleft} t $ in $T_*$ (and $|\suc_{T_*}(t)|\gg b_*(t)$).

The following is the central definition of this work.

\begin{definition}[{\cite[Def.~1.14]{KST}}]\label{def:Etilde}
    Let $\tbf$ be a tree-creature frame as in \autoref{def:treeframe}. 
    Define the poset $\Etld:= \Etld_{\tbf}$ whose conditions are subtrees $p\subseteq T_*$ such that, for each $t\in p$ above $\stem(p)$, $\|p\|_t:=\|\suc_p(t)\|_t \geq 1+ \frac{1}{|\stem(p)|}$ (we interpret $\frac{1}{0} = \infty$). We order $\Etld$ by $\subseteq$.
    
    Note that $T_*\in \Etld$ and that it is the maximum condition of $\Etld$. 
    Also, whenever $p\in\Etld$ and $t\in p$ is above the stem, $p|^t \in \Etld$ (the subtree of $p$ whose nodes are precisely the nodes in $p$ compatible with $t$) and it is stronger than $p$.
\end{definition}

Conditions with stem $\la\ \ra$ are somewhat uninteresting.

\begin{fact}\label{cent}
Whenever $p\in \Etld$ and $\stem(p) = \la\ \ra$, $\|p\|_t =\infty$ for all $t\in p$. 
As a consequence, the set of conditions with stem $\la\ \ra$ is centered.
\end{fact}
\begin{proof}
The first part follows by \autoref{def:Etilde} and the convention $\frac{1}{0} = \infty$. The second part follows by \autoref{lem:mbig}: when $I$ is finite and $\la r_i:\, i\in I\ra$ is a sequence of conditions with empty stem, $r:=\bigcap_{i\in I}r_i$ is in $\Etld$ (even with empty set) because, for any $t\in r$,
\[\left\|\bigcap_{i\in I}\suc_{r_i}(t)\right\|_t \geq \infty - \log_{b_*(t)}|I| = \infty.\qedhere\]
\end{proof}

It is not hard to show that $\Etld$ has the ccc. Moreover, it is $\sigma$-$k$-linked for any $2\leq k<\omega$. Recall that, given a poset $\Por$, $Q\subseteq \Por$ is \emph{$k$-linked} if any subset of $Q$ of size ${\leq} k$ has a lower bound in $\Por$; and the poset $\Por$ is \emph{$\sigma$-$k$-linked} if it can be covered by countably many $k$-linked subsets.

\begin{lemma}\label{klinked}
  Let $\tbf$ be a tree-creature frame. Then, for any $2\leq k<\omega$:
  \begin{enumerate}[label=\normalfont(\alph*)]
	\item\label{klinkeda} For each $s\in T_*$, the set $\displaystyle E_{s,k} := \lset{p\in \Etld}{\stem(p)=s \text{ and } \|p\|_t \geq 1 + \frac{1}{|s|}+\log_{b_*(t)}k}$ is $k$-linked in $\Etld:= \Etld_\tbf$.
	\item\label{klinkedb} $\bigcup \set{E_{s,k}}{s\in T_*,\ k<\omega}$ is dense in $\Etld$. In particular, $\Etld$ is $\sigma$-$k$-linked.
  \end{enumerate}
\end{lemma}
\begin{proof}\ \\
   \ref{klinkeda}: If $\set{p_i}{i<k}\subseteq E_{s,k}$, let $p:=\bigcap_{i<k}p_i$. This is a subtree of $T_*$ with $s\in p$ and, for $t\in p$ above $s$, by \autoref{lem:mbig},
   \[\|p\|_t \geq 1 + \frac{1}{|s|}+\log_{b_*(t)}k - \log_{b_*(t)}k = 1+\frac{1}{|s|},\]
   so $p\in \Etld$. This shows that $E_{s,k}$ is $k$-linked.
   
   \ref{klinkedb}: Let $p\in \Etld$ and $s_0:=\stem(p)$. By~\ref{tree2}, find $s\in p$ (above $s_0$) such that $\frac{1}{|s|}+\log_{b_*(t)} k < \frac{1}{|s_0|}$ for all $t\in T_*$ above $s$. Then, $p|^s \in E_{s,k}$ and it is stronger than $p$.
\end{proof}

\section{Fam-limits}\label{sec:famlim}

From now on, we fix a tree-creature frame $\tbf$ as in \autoref{def:treeframe}, and let $\Etld:= \Etld_\tbf$. We also fix the following assumption (stronger than~\ref{tree2}) until the end of this paper.

\begin{assumption}\label{hyplog}
$\displaystyle \lim_{n\to \infty} \min_{t\in \Lev_n(T_*)} \log_n b_*(t) = \infty$. Equivalently, the function 
\[n\mapsto \min_{t\in\Lev_n(T_*)}b_*(t)\] 
dominates the set $\set{m^\id}{m<\omega}$, where $m^\id$ is the function sending $n\mapsto m^n$.
\end{assumption}

Based on~\cite[Lem.~1.20]{KST}, 
we plan to show that $\Etld=\Etld_\tbf$ has (strong) fam-limits under the previous assumption. We first review the formalization of the notion of \emph{strong fam-limits} from \cite{Uribethesis,CMU}.

\begin{definition}\label{def:fam}
   Let $\Bor$ be a Boolean algebra. A \emph{finitely additive measure (fam) on $\Bor$} is a map $\Xi\colon \Bor\to[0,\infty]$ satisfying
   \begin{enumerate}[label=\normalfont (\roman*)]
    \item \label{m4a} $\Xi (0_{\Bor})=0$,
            
    \item \label{m4b} $\Xi(a\vee b)=\Xi(a)+\Xi(b)$ whenever $a,b\in\Bor$ and $a \wedge b= 0_{\Bor}$.
   \end{enumerate}
   If in addition $\Xi(1_\Bor)=1$, we say that $\Xi$ is a \emph{probability fam}. From now on, we will assume that all our fams are of probability.
  
   Denote by $\Pbf^{\Bor}$ the collection of \emph{finite partitions of $1_\Bor$}, i.e. $P\in \Pbf^{\Bor}$ iff $P\subseteq\Bor$ is finite, $a\wedge b=0_\Bor$ for $a\neq b$ in $P$, and $\bigvee P = 1_\Bor$. When $\Xi$ is a fam on $\Bor$, we also write $\Pbf^\Xi$ for $\Pbf^\Bor$.
   
   Let $K$ be a non-empty set. Recall that a \emph{field of sets over $K$} is a subalgebra of $\pts(K)$ (under the set-theoretic operations). When $\Bor$ is a field of sets over $K$ and $\Xi$ is a fam on $\Bor$, we say that $\Xi$ is \emph{free} if, for all $k\in K$, $\{k\}\in \Bor$ and $\Xi(\{k\}) = 0$ (this implies that all finite subsets of $K$ have measure zero).
\end{definition}

We use the following particular type of fams.

\begin{definition}[{\cite{CMUP}}]\label{uap}
  Let $\Bor$ be a field of sets over $K$ and $\Xi$ a (probability) fam on $\Bor$. We say that $\Xi$ has the \emph{uniform approximation property (uap)} if, for any $\varp>0$ and any $P\in \Pbf^\Xi$, there is some non-empty finite $u\subseteq K$  such that, for all $B\in P$,
        \[\left|\frac{|u\cap B|}{|u|} - \Xi(B)\right| < \varp.\]
\end{definition}

\begin{lemma}[{\cite{CMUP}}]\label{freeuap}
  Any fam over a field of sets $\Bor$ over $K$ satisfying that all finite sets in $\Bor$ have measure zero has the uap. In particular, any free fam has the uap.
\end{lemma}

Surprisingly, fams with the uap having a finite set of positive measure are easily characterized. Concretely, for any such fam $\Xi$ on $\Bor$ there is some natural number $d>0$ such that, for any $B\in\Bor$, $\Xi(B)$ has the form $\frac{\ell}{d}$ for some $0\leq \ell \leq d$, and $\ell\leq |B|$ when $B$ is finite. Details will be available in~\cite{CMUP}.

For fams with the uap, the size of the set $u$ has a bound that depends only on $\varp$ and $|P|$.

\begin{lemma}[{\cite[Lem.~1.2]{KST}, \cite{CMUP}}]\label{ubound}
Let $\Xi$ be a fam with the uap on a field of sets $\Bor$ over $K$. Then, for any $\varp>0$ and $m\in\omega$, there is some $M:= M_{\varp,m}\in\omega$ such that, for any $P \in \Pbf^\Xi$, if $|P|\leq m$ then the $u$ in \autoref{uap} can be found of size ${\leq}M$.
\end{lemma}

We are ready to introduce our formalization of forcings with fam-limits.

\begin{definition}[{\cite{Uribethesis,CMU}}]\label{def:famlim}
    Let $\Por$ be a poset.
    \begin{enumerate}[label = \normalfont (\arabic*)]
        \item\label{faml1} Let $\Xi\colon \pts(K) \to [0,1]$ be a fam with the uap, $\bar I = \seq{I_k}{k\in K}$ a partition of a set $W$ into finite sets, and $\varp>0$.

        A set $Q\subseteq\Por$ is \emph{$(\Xi,\bar I,\varp)$-linked} if there is a function $\lim\colon Q^W\to \Por$ and a $\Por$-name $\dot \Xi'$ of a fam with the uap on $\pts(K)$ extending $\Xi$ such that, for any $\bar p = \seq{ p_\ell}{\ell\in W} \in Q^W$,
        \begin{equation}\label{intext}
        \lim \bar p \Vdash \int_K \frac{|\set{\ell \in I_k}{p_\ell \in \dot G}|}{|I_k|}d\dot \Xi' \geq 1-\varp.
        \end{equation}

        \item\label{faml2} Let $\mu$ be an infinite cardinal, and let $\calY\subseteq\calY_*$, where $\calY_*$ is the class of all pairs $(\Xi,\bar{I})$ such that $\Xi$ is a fam with the uap on some $\pts(K)$ and $\bar{I} = \seq{I_k}{ k\in K}$ is a pairwise disjoint family of finite non-empty sets.\footnote{Notice that $K$ is not a fixed set.}
        
        The poset $\Por$ is \emph{$\mu$-$\calY$-linked}, witnessed by $\seq{Q_{\alpha,\varp}}{\alpha<\mu,\ \varp\in(0,1)\cap \Q}$, if:
        \begin{enumerate}[label = \normalfont (\roman*)]
            \item Each $Q_{\alpha,\varp}$ is $(\Xi,\bar I,\varp)$-linked for any $(\Xi,\bar I) \in \calY$.
            \item\label{faml2d} For $\varp\in(0,1)\cap \Q$, $\bigcup_{\alpha<\omega} Q_{\alpha,\varp}$ is dense in $\Por$.
        \end{enumerate}

        \item\label{faml3} The poset $\Por$ is \emph{uniformly $\mu$-$\calY$-linked} if there is some $\seq{Q_{\alpha,\varp}}{\alpha<\mu,\ \varp\in(0,1)\cap \Q}$ as above, such that in~\ref{faml1} the name $\dot \Xi'$ only depends on $(\Xi,\bar I)$ (and not on any $Q_{\alpha,\varp}$, although we may have different limits on each $Q_{\alpha,\varp}$).
    \end{enumerate}
    We write \emph{$\sigma$-$\calY$-linked} when $\mu=\aleph_0$.
\end{definition}

\begin{example}\label{exm:famlk}
    \ 
    \begin{enumerate}[label = \normalfont (\arabic*)]
        \item\label{famsing} Any singleton is $(\Xi,\bar I,\varp)$-linked: If $Q$ is a singleton, $Q^W$ only contains one constant sequence, so $\lim \bar p$ can be defined as this constant value. Notice that $\lim \bar p$ forces that the integral of \autoref{intext} is $1$ for any $\dot \Xi'$ extending $\Xi$.      
        Hence, any poset $\Por$ is uniformly $|\Por|$-$\calY_*$-linked. In particular, Cohen forcing is uniformly $\sigma$-$\calY_*$-linked.

        \item\label{randomfam} Shelah~\cite{ShCov} proved, implicitly, that random forcing is uniformly $\sigma$-$\calY_{**}$-linked, where $\calY_{**}$ is the class of all $(\Xi,\bar I)\in \calY_*$ such that $\Xi$ is free.  
        More generally, we proved that any measure algebra with Maharam type $\mu$ is uniformly $\mu$-$\calY_{**}$-linked~\cite{MUrandom}.
        
        \item\label{famlkE} With Cardona and Uribe-Zapata we have proved that $\sigma$-$\calY_*$-linked posets do not increase $\non(\Ewf)$, where $\Ewf$ denotes the ideal generated by the $F_\sigma$ measure zero subsets of $2^\omega$, and that $\Ebb$ (the standard ccc poset adding an eventually different real) and localization posets increase $\non(\Ewf)$ (see~\cite{CardonaRIMS,CMlocalc}). As a consequence, $\Ebb$ and the localization posets cannot be $\sigma$-$\calY_*$-linked.
    \end{enumerate}
\end{example}

The following is the main result of this section.

\begin{theorem}\label{mainEtld}
   Under \autoref{hyplog}, the poset $\Etld$ is uniformly $\sigma$-$\calY_*$-linked.
\end{theorem}

Before proving the theorem, we review some results about fams. Given a fam $\Xi$ on a field of sets $\Bor$ over $K$, and a bounded function $\Xi\colon \Bor\to \R$, we can define the \emph{$\Xi$-integral of $f$}, denoted $\int_K f d\Xi$, similarly as the Riemann integral (using $\Pbf^\Xi$, lower sums and upper sums), and we say that $f$ is \emph{$\Xi$-integrable} when its $\Xi$-integral is defined. When $\Bor = \pts(K)$, any bounded real-valued function on $K$ is $\Xi$-integrable. For details, see~\cite{Uribethesis,CMUP}.

\begin{theorem}[{\cite{CMUP}}]\label{t85}
    Let $\Xi_0$ be a fam with the uap on a field of sets $\Bor$ over $K$, and let 
    $I$ be a set. For each $i\in I$, let $C_i$ be a closed subset of $\R$ and $f_i\colon K\to\R$ bounded.  
   Then, the following statements are equivalent.
   \begin{enumerate}[label = \normalfont (\Roman*)]
   \item\label{t85I} For any $P\in\Pbf^{\Xi_0}$, $\varepsilon>0$, any finite set $J\subseteq I$, and any open $G_i\subseteq \R$ containing $C_i$ for $i\in J$, there is some non-empty finite $u\subseteq K$ such that:
   \begin{enumerate}[label=\normalfont(\roman*)]
       \item\label{t85a} $\displaystyle \left |\frac{|B\cap u|}{|u|} - \Xi_0(B) \right | <\varepsilon$ for any $B\in P$, and
       
       \item\label{t85b} $\displaystyle \frac{1}{|u|}\sum_{k\in u}f_i(k) \in  G_i$ for any $i\in J$.
   \end{enumerate}
   \item\label{t85II} There is some fam $\Xi$ on $\pts(K)$ with the uap extending $\Xi_0$ such that, for any $i \in I$, $\displaystyle \int_{K} f_id \Xi \in C_i$.
   \end{enumerate}
\end{theorem}

The previous result allows the following characterization.

\begin{theorem}[{cf.~\cite{Uribethesis,CMU}}]\label{char}
 Let $\mu$ be a cardinal, $\Por$ a poset, $\seq{Q_{\alpha,\varp}}{\alpha<\mu,\ \varp\in(0,1)\cap\Q}$ a sequence of subsets of $\Por$, $\calY\subseteq\calY_*$ and, for each $(\Xi,\bar I)\in \calY$, $\alpha<\mu$ and $\varp\in(0,1)\cap\Q$, ${\lim}^{\Xi,\bar I}\colon Q^W_{\alpha,\varp}\to \Por$ where $W:=\bigcup_{k\in K} I_k$. Then, the following statements are equivalent.
\begin{enumerate}[label=\normalfont (\Roman*)]
\item\label{char1} $\Por$ is uniformly $\mu$-$\calY$-linked witnessed by $\seq{Q_{\alpha,\varp}}{\alpha<\mu,\ \varp\in(0,1)\cap\Q}$.
\item\label{char2} \autoref{def:famlim}~\ref{faml2}~\ref{faml2d} holds and, for any
    \begin{multicols}{2} 
    \begin{itemize}
        \item $(\Xi,\bar I)\in\calY$,
        \item $i^*<\omega$,
        \item $(\alpha_i,\varp_i)\in \mu\times ((0,1)\cap \Q)$,
        \item $\bar r^i = \seq{r^i_\ell}{\ell\in W} \in Q_{\alpha_i,\varp_i}^W$ for $i<i^*$,
        \item $P\in \Pbf^\Xi$,
        \item $\varp'>0$, and
        \item $q\in \Por$ stronger than $\lim^{\Xi,\bar I} \bar r^i$ for all $i<i^*$,
    \end{itemize}
    \end{multicols}
    there is some $q'\leq q$ in $\Por$ and $u\subseteq K$ finite non-empty such that
    \begin{enumerate}[label = \normalfont (\arabic*)]
        \item\label{-fm1} $\displaystyle \left| \frac{|u\cap B|}{|u|} - \Xi(B)\right| < \varp'$ for all $B\in P$, and
        \item\label{-fm2} $\displaystyle \frac{1}{|u|}\sum_{k\in u}\frac{|\set{\ell\in I_k}{q' \leq r^i_\ell}|}{|I_k|} > 1-\varp_i-\varp'$ for all $i<i^*$.
    \end{enumerate}
\end{enumerate}
\end{theorem}
\begin{proof}
\ref{char1}${}\Rightarrow{}$\ref{char2}: Assume~\ref{char1} and the assumptions of~\ref{char2}. Let $\dot\Xi'$ be as in \autoref{def:famlim}~\ref{faml3}, and let $G$ be $\Por$-generic over $V$ such that $q\in G$, and $\Xi':=\dot\Xi'[G]$. In $V[G]$, since $q\leq \lim^{\Xi,\bar I}\bar r^i$ for all $i<i^*$, by \autoref{intext}
\[\int_K f_i d\Xi' \geq 1-\varp_i, \text{ where }f_i(k):=\frac{|\set{\ell \in I_k}{r^i_\ell \in \dot G}|}{|I_k|}.\]
Then, by \autoref{t85}, there is some non-empty finite $u\subseteq K$ such that
\begin{enumerate}[label=\normalfont(\roman*)]
       \item $\displaystyle \left |\frac{|B\cap u|}{|u|} - \Xi'(B) \right | <\varepsilon'$ for any $B\in P$, and
       
       \item $\displaystyle \frac{1}{|u|}\sum_{k\in u}f_i(k) > 1-\varp_i -\varp'$ for any $i<i^*$.
   \end{enumerate}
   Since $\Xi'(B)=\Xi(B)$ for all $B\in\pts(K)\cap V$, we obtain~\ref{-fm1}.
   
   Back in $V$, find $q'\leq q$ forcing the above and such that either $q'\leq r^i_\ell$ or $q'\perp r^i_\ell$ for all $i<i^*$ and $\ell \in \bigcup_{k\in u}I_k$. Then, $q'$ decides the value of $f_i(k)$ for all $i<i^*$ and $k\in u$, even more, $q'$ forces
   \[f_i(k) = \frac{\set{\ell\in I_k}{q'\leq r^i_\ell}}{|I_k|}.\]
   Then,~\ref{-fm2} follows.
   
   \ref{char2}${}\Rightarrow{}$\ref{char1}: Assume~\ref{char2} and $(\Xi,\bar I)\in \calY$. Let $G$ be $\Por$-generic over $V$ and work in $V[G]$. Consider the set
   \[\Omega := \lset{(\alpha,\varp,\bar r)}{\alpha<\mu,\ \varp\in(0,1)\cap\Q,\ \bar r\in Q^W_{\alpha,\varp}\cap V,\ {\lim}^{\Xi,\bar I}\bar r \in G}.\]
   
   For each $\alpha<\mu$, $\varp\in(0,1)\cap \Q$ and $\bar r\in Q^W_{\alpha,\varp}\cap V$, define
   \[f_{\alpha,\varp,\bar r}(k):= \frac{\set{\ell \in I_k}{r_\ell\in \dot G}}{|I_k|},\qquad K_{\alpha,\varp,\bar r}:=[1-\varp,\infty).\]
   To prove~\ref{char1}, it is enough to check~\autoref{t85}~\ref{t85I} for $\Omega$. Indeed, assume
   \begin{multicols}{2} 
    \begin{itemize}
        \item $i^*<\omega$,
        \item $(\alpha_i,\varp_i,\bar r^i)\in \Omega$ for $i<i^*$,
        \item $P\in \Pbf^\Xi$, and
        \item $\varp'>0$.
    \end{itemize}
    \end{multicols}
    Back in $V$, let $q\in \Por$ be a condition forcing the above such that, wlog, $q\leq \lim^{\Xi,\bar I}\bar r^i$ for all $i<i^*$. Then, by~\ref{char2}, there exists some $q'\leq q$ in $\Por$ and a finite non-empty $u\subseteq K$ satisfying~\ref{-fm1} and~\ref{-fm2}. This density argument allows to find such a $q'$ in $G$. Therefore, in $V[G]$, for any $i<i^*$,
    \[\frac{1}{|u|}\sum_{k\in u}f_{\alpha_i,\varp_i,\bar r^i}(k) \geq \frac{1}{|u|}\sum_{k\in u}\frac{|\set{\ell\in I_k}{q' \leq r^i_\ell}|}{|I_k|} > 1-\varp_i-\varp'. \qedhere\]
\end{proof}

We now proceed to prove \autoref{mainEtld}. Some preparation is needed to define the correct limit function and the witness for the uniform $\sigma$-$\calY_*$-linkedness.

\begin{definition}\label{E's}
For $s\in T_*\menos\{\la\ \ra\}$, define
    \begin{align*}
       E'_s & := \lset{p\in \Etld}{\stem(p) = s \text{ and }\|p\|_t \geq 1+\frac{1}{|s|} + 4 \log_{b_*(t)}|t| \text{ for all $t\in p$ above $s$}},\\
       E' & := \bigcup_{s\in T_*\menos\{\la\ \ra\}}E'_s.
	 \end{align*}
	 Note that,  
	 by \autoref{hyplog}, $E'$ is dense in $\Etld$ (which can be proved similarly as \autoref{klinked}).
	 \end{definition}

In the following lemma, we show how to homogenize finitely many conditions in $\Etld$.

\begin{lemma}[{cf.~\cite[Lem.1.20, Step~1]{KST}}]\label{psf}
Let $s\in T_*\menos\{\la\ \ra\}$, $I$ a finite set and $\set{r_\ell}{\ell\in I}\subseteq E'_{s}$. Then there is a condition $r^*\in\Etld$ with stem $s$ such that
\begin{enumerate}[label = \normalfont (\roman*)]
        \item\label{bul1} $\displaystyle \|\suc_{r^*}(t)\|\geq 1+\frac{1}{|s|} + 2 \log_{b_*(t)}|t|$ for any $t\in r^*$ above $s$, and

        \item\label{bul2} 
        The set $\displaystyle \lset{r\in \Etld}{\frac{\set{\ell \in I}{r\leq r_\ell}|}{|I|} > 1-\frac{1}{|s|}}$ is dense below $r^*$.
    \end{enumerate}
\end{lemma}

The condition $r^*$ constructed in the proof is called the \emph{pseudo-fusion of $\set{r_\ell}{\ell\in I}$}.

\begin{proof}
For $n\geq |s|$, denote
    \[\delta_{n}:=\left\{
       \begin{array}{ll}
           \Big(1-\frac{1}{|s|}\Big)\Big(1+\frac{1}{n-1}\Big) & \text{if $n>|s|$,} \\[2ex]
           1 & \text{if $n=|s|$.}
       \end{array}
    \right.\]

For $t\in T_*$ above $s$ and $n\geq |s|$, define 
    \begin{align*}
        J_{t} & :=\set{\ell\in I}{t\in r_\ell},\\
        L_{n} & :=\lset{t\in T_*}{|t|=n,\ \frac{|J_{t}|}{|I|}\geq \delta_{n}}.
    \end{align*}
    We define $r^*$ by recursion on the height such that its $n$-th level is contained in $L_n$ for $n\geq |s|$: up to level $|s|$, $r^*$ is determined by $s$ (i.e. just the stem); and, when $n\geq |s|$ and $r^*$ is defined up to height $n$, for each $t\in r^*$ at level $n$ we set $\suc_{r^*}(t) := \suc_{T_*}(t)\cap L_{n+1}$.     
    Apply $b_*(t)$-co-bigness to $\seq{\suc_{r_\ell}(t)}{\ell\in J_t}$ and $\ell\in J_t\mapsto \frac{1}{|J_t|}$ to obtain
    \begin{align*}
       \left\|\lset{t'\in\suc_{T_*}(t)}{|J_{t'}| \geq |J_{t}|\left(1-\frac{1}{n^2}\right)}\right\|_t & \geq 1 + \frac{1}{|s|} + 4\log_{b_*(t)}|t| - \log_{b_*(t)}n^2\\ 
         &= 1+\frac{1}{|s|} + 2 \log_{b_*(t)}|t|.
    \end{align*}
    Since $t\in L_n$, any $t'$ in this set satisfies
    \[|J_{t'}| \geq |J_t|\left(1-\frac{1}{n^2}\right)\geq |I|\delta_{n}\left(1-\frac{1}{n^2}\right) = |I|\delta_{n+1},\]
    so $t'\in L_{n+1}$, i.e.\ $t'\in \suc_{r^*}(t)$. Therefore, 
    $\|\suc_{r^*}(t)\|_t \geq 1+\frac{1}{|s|} + 2 \log_{b_*(t)}|t|$ which establishes~\ref{bul1}.

    We know show~\ref{bul2}. 
    Let $r\leq r^*$. By strengthening $r$ if necessary, by \autoref{klinked} we may assume that $r\in E_{t,|I|+1}$ for some $t\in r_*$ above $s$ of length ${\geq}|I|+1$. For $\ell\in J_t$, $t\in r_\ell$ and $r_\ell\in E'_s$, so $r_\ell|^t \in E_{t,|I|+1}$. Then there is some $r'\in \Etld$ stronger than $r$ and $r_\ell$ for all $\ell \in J_t$ because $E_{t,|I|+1}$ is $(|I|+1)$-linked. On the other hand, $t\in L_{|t|}$, so $\frac{|J_t|}{|I|}\geq \delta_{|t|}$. Therefore, 
    \[\frac{|\set{\ell \in I}{r'\leq r_\ell}|}{|I|} \geq \frac{|J_t|}{|I|}\geq  \delta_{|t|} > 1-\frac{1}{|s|}. \qedhere\]
\end{proof}

Using the previous result, we show how to define fam-limits on $\Etld$.

\begin{theorem}[{cf.~\cite[Lem.~1.20]{KST}}]\label{fam-lim1}
    Let $\Xi\colon \pts(K)\to [0,1]$ be a fam with the uap and $\bar I = \seq{I_k}{k\in K}$ a partition of a set $W$ into finite sets. Then,   
    for any $s\in T_*\menos\{\la\ \ra\}$ there is a function $\lim^{\Xi,\bar I}\colon (E'_{s})^W\to \Etld$ such that $\lim^{\Xi,\bar I}\bar r$ has stem $s$ for any $\bar r \in (E'_{s})^W$, and satisfying: For any
    \begin{multicols}{2} 
    \begin{itemize}
        \item $i^*<\omega$,
        \item $s_i\in T_*\menos\{\la\ \ra\}$,
        \item $\bar r^i = \seq{r^i_\ell}{\ell\in W} \in (E'_{s_i})^W$ for $i<i^*$,
        \item $P\in \Pbf^\Xi$,
        \item $\varp'>0$, and
        \item $q\in \Etld$ stronger than $\lim^{\Xi,\bar I} \bar r^i$ for all $i<i^*$,
    \end{itemize}
    \end{multicols}
    there is some $q'\leq q$ in $\Etld$ and $u\subseteq K$ finite non-empty such that
    \begin{enumerate}[label = \normalfont (\arabic*)]
        \item\label{fm1} $\displaystyle \left| \frac{|u\cap B|}{|u|} - \Xi(B)\right| < \varp'$ for all $B\in P$, and
        \item\label{fm2} $\displaystyle \frac{1}{|u|}\sum_{k\in u}\frac{|\set{\ell\in I_k}{q' \leq r^i_\ell}|}{|I_k|} > 1-\frac{2}{|s_i|}-\varp'$ for all $i<i^*$.
    \end{enumerate}
\end{theorem}
\begin{proof}
Let $s\in T_*$. We first show how to define $\lim^{\Xi,\bar I}\bar r$ for $\bar r = \seq{r_\ell}{\ell\in W}\subseteq (E'_{s})^W$. For each $k\in K$, let $r^*_k$ be the pseudo-fusion of $\seq{r_\ell}{\ell\in I_k}$ as in \autoref{psf}. For each $t\in T_*$ set $Z_t= Z^{\bar r}_t:= \set{k\in K}{t\in r^*_k}$. Note that $Z_s = K$. By recursion on the height, we construct $r^* :=\lim^{\Xi,\bar I}\bar r$ in $\Etld$ with stem $s$ such that, for any $t\in r^*$ above $s$,
$\Xi(Z_t) \geq \delta_{|t|}$, where $\delta_n$ for $n\geq |s|$ is as in the proof of \autoref{psf}.
    Up to $|s|$, $r^*$ is determined by $s$. Let $n\geq|s|$ and assume we have constructed $r^*$ up to height $n$. It is enough to show how $\suc_{r^*}(t)$ is defined for any $t\in r^*$ at level $n$. Enumerate the finite set $\set{\suc_{r^*_k}(t)}{k\in Z_t}$ as $\set{a_j}{j<m}$. Consider the function $g\colon m\to [0,1]$ defined by 
    \[g(j):= \frac{\Xi(\set{k\in Z_t}{\suc_{r^*_k}(t) = a_j})}{\Xi(Z_t)}\]
    Then, by $b_*(t)$-co-bigness, 
    \[\left\|\lset{t'\in \suc_{T_*}(t)}{\frac{\Xi(Z_{t'})}{\Xi(Z_t)} \geq 1-\frac{1}{n^2}}\right\|_t \geq 1+\frac{1}{|s|} + 2 \log_{b_*(t)}|t|- \log_{b_*(t)}n^2 = 1 + \frac{1}{|s|},\]
    so define $\suc_{r^*}(t)$ as this set. On the other hand, for $t'\in \suc_{r^*}(t)$,
    \[\Xi(Z_{t'}) \geq \Xi(Z_t)\left(1-\frac{1}{n^2}\right) \geq \delta_n \left(1-\frac{1}{n^2}\right) = \delta_{n+1}.\]
    It is clear that $r^*\in \Etld$ and $\stem(r^*) = s$.
    
    We now prove that this limit works. Work under the assumption for~\ref{fm1} and~\ref{fm2}. For $i<i^*$ and $k\in K$, let $r^i_k$ be the pseudo-fusion of $\seq{r^i_\ell}{\ell\in I_k}$, $r^i:=\lim^{\Xi,\bar I}\bar r^i$ and $Z^i_t:= Z^{\bar r^i}_t$. So we have that $q\leq r^i$ for all $i<i^*$.

    Let $t\in q$ above the stem. By \autoref{ubound} applied to the atoms of the field generated by $P\cup \set{Z^i_t}{i<i^*}$ (there are at most $|P|2^{i^*}$ many), there is some non-empty finite set $u_t\subseteq W$ satisfying~\ref{fm1} and, for $i<i^*$,
    \[\frac{|Z^i_t\cap u_t|}{|u_t|}> \delta^i_{|t|}-\varp'>1-\frac{1}{|s_i|}-\varp'\]
    where $\delta^i_n$ is as $\delta_n$ in the proof of \autoref{psf} for $s_i$ and $n\geq|s_i|$. 
    Moreover, there is some $M>0$ such that $|u_t|\leq M$ for all $t$ as above (concretely, $M:=M_{\varp',|P|2^{i^*}}$ as in \autoref{ubound}). So pick $t\in q$ of large enough length such that $q|^t \in E_{t, M i^* +1}$, and set $u:=u_t$. Then, by \autoref{klinked}, there is a lower bound $q_1$ of $\{q\}\cup\{r^i_k\colon i<i^*,\ k\in Z^i_t\cap u\}$. By using \autoref{psf}~\ref{bul2} $\left|\bigcup_{i<i^*}Z^i_t\cap u\right|$-many times, we can find some $q'\leq q_1$ such that, for any $i<i^*$ and $k\in Z^i_t\cap u$,
    \[\frac{|\set{\ell\in I_k}{q'\leq r^i_\ell}|}{|I_k|}> 1- \frac{1}{|s_i|}.\]
    Then, for $i<i^*$,
    \begin{align*}
       \frac{1}{|u|}\sum_{k\in u}\frac{|\set{\ell\in I_k}{q' \leq r^i_\ell}|}{|I_k|} & > \frac{|Z^i_t \cap u|}{|u|}\left(1-\frac{1}{|s_i|}\right) > \left(1-\frac{1}{|s_i|}-\varp'\right)\left(1-\frac{1}{|s_i|}\right)\\ 
        & = 1- \frac{2}{|s_i|}-\varp' + \frac{1}{|s_i|^2}+ \frac{\varp'}{|s_i|} > 1- \frac{2}{|s_i|}-\varp'.\qedhere
 	 \end{align*}
\end{proof}

\autoref{mainEtld} follows directly by the previous theorem.

\begin{proof}[Proof of \autoref{mainEtld}]
For $s\in T_*\menos\{\la\ \ra\}$ and $\varp \in (0,1)\cap \Q$, define
\[Q_{s,\varp}:=
\left\{
\begin{array}{ll}
E'_s & \text{if $\displaystyle \frac{2}{|s|}\leq \varp$,}\\[3ex]
\{T_*|^s\} & \text{otherwise.}
\end{array}
\right.\]
It is easy to show that, for any $\varp \in (0,1)\cap \Q$,
\[\bigcup_{s\in T_*\menos\{\la\ \ra\}}Q_{s,\varp} \supseteq \bigcup\lset{E'_s}{s\in  T_*\menos\{\la\ \ra\},\ |s|\geq \frac{2}{\varp}} \text{ is dense in }\Etld.\]

By checking~\ref{char2} of \autoref{char}, 
we show that $\seq{Q_{s,\varp}}{s\in T_*\menos\{\la\ \ra\},\ \varp\in(0,1)\cap \Q}$ witnesses that $\Etld$ is $\sigma$-$\calY_*$-linked. Assume
\begin{multicols}{2} 
    \begin{itemize}
        \item $(\Xi,\bar I)\in\calY_*$,
        \item $i^*<\omega$,
        \item $(s_i,\varp_i)\in (T_*\menos\{\la\ \ra\})\times ((0,1)\cap \Q)$,
        \item $\bar r^i = \seq{r^i_\ell}{\ell\in W} \in Q_{s_i,\varp_i}^W$ for $i<i^*$,
        \item $P\in \Pbf^\Xi$,
        \item $\varp'>0$, and
        \item $q\in \Etld$ stronger than $\lim^{\Xi,\bar I} \bar r^i$ for all $i<i^*$.
    \end{itemize}
    \end{multicols}
    When $\varp_i<\frac{2}{|s_i|}$ we are dealing with the singleton $Q_{s_i,\varp_i}=\{T_*|^{s_i}\}$, for which the sequence $\bar r^i$ is constant and $\lim^{\Xi,\bar I}\bar r^i := T_*|^{s_i}$, so $q \leq T_*|^{s_i}$. Now, apply \autoref{fam-lim1} to those $i<i^*$ such that $\frac{2}{|s_i|}\leq \varp_i$, and find $q'\leq q$ in $\Etld$ and a finite non-empty $u\subseteq K$ satisfying~\ref{fm1} and~\ref{fm2} (for those $i$). Then, whenever $\frac{2}{|s_i|}\leq \varp_i$,
    \[\frac{1}{|u|}\sum_{k\in u}\frac{|\set{\ell\in I_k}{q' \leq r^i_\ell}|}{|I_k|} > 1-\frac{2}{|s_i|}-\varp' \geq 1-\varp_i -\varp'.\]
    On the other hand, whenever $\varp_i<\frac{2}{|s_i|}$,
    \[\frac{1}{|u|}\sum_{k\in u}\frac{|\set{\ell\in I_k}{q' \leq r^i_\ell}|}{|I_k|} = 1 > 1-\varp_i-\varp'.\qedhere\]
\end{proof}

\begin{remark}
In~\cite{KST}, the function $\loss\colon E'\to \Q$ is used most of the time, but it is not essential as seen above. Their definition of $\loss$, adapted to this paper, is basically $\loss(p)=\frac{2}{|\stem(p)|}$. Hence, in the proof of \autoref{mainEtld}, $Q_{s,\varp}$ refers to the set of $p\in E'$ with stem $s$ such that $\loss(p)\leq \varp$ (note that there are no such conditions when $\varp < \frac{2}{|s|}$).
\end{remark}

\begin{remark}
In~\cite{Uribelinked}, Uribe-Zapata defined the notion \emph{$\mu$-intersection-linked} for posets, where $\mu$ is an infinite cardinal, using the intersection number from Kelley~\cite{Kelley}. Strictly speaking, this property should be part of \autoref{def:famlim}~\ref{faml2} ($\mu$-$\calY$-linked) for $\seq{Q_{\alpha,\varp}}{\alpha<\mu,\ \varp\in(0,1)\cap\Q}$, so that FS iterations of such posets have fam-limits~\cite{Uribethesis,CMU}, but we excluded it for practicality. Moreover, we proved that, whenever $Q\subseteq\Por$ is $(\Xi,\bar I,\varp)$-linked, $K=\omega$ and $\lim_{k\to\infty}|I_k| = \infty$, $Q$ has intersection number ${\geq}1-\varp$~\cite{CMU}. For this reason, we obtain this condition about the intersection number for free in many cases, e,g.\ for measure algebras and $\Etld$. But note that, for the later, \autoref{psf} implies that $E'_s$ has intersection number ${\geq}1-\frac{1}{|s|}$.
\end{remark}

\section{Uf-limits on intervals}\label{sec:uflim}

Recall that an ultrafilter on a Boolean algebra can be seen as a fam taking values in $\{0,1\}$. In this sense, any ultrafilter has the uap.

We present a version of~\autoref{def:famlim} for ultrafilters,\footnote{This may not equivalent to \autoref{def:famlim} for fams with values in $\{0,1\}$, since they may not be extended to an ultrafilter (but to some fam) in the generic extension.} which is the notion we call \emph{ultrafilter-limits for intervals} in the Introduction.

\begin{definition}\label{def:uflim}
Let $\Por$ be a poset.
    \begin{enumerate}[label = \normalfont (\arabic*)]
        \item\label{ul1} Let $D$ be an ultrafilter on $\pts(K)$ for some non-empty set $K$, $\bar I = \seq{I_k}{k\in K}$ a partition of a set $W$ into finite sets, and $\varp>0$.

        A set $Q\subseteq\Por$ is \emph{$(D,\bar I,\varp)^*$-linked}\footnote{We add the $*$ to avoid confusion with \autoref{def:famlim} when $D$ is intepreted as a fam.} if there is a function $\lim\colon Q^W\to \Por$ and a $\Por$-name $\dot D'$ of an ultrafilter on $\pts(K)$ extending $D$ such that, for any $\bar p = \seq{ p_\ell}{\ell\in W} \in Q^W$,
        \begin{equation}\label{ufext}
        \lim \bar p \Vdash \lset{k\in K}{\frac{|\set{\ell \in I_k}{p_\ell \in \dot G}|}{|I_k|}\geq 1-\varp} \in \dot D'.
        \end{equation}

        \item\label{ul2} Let $\mu$ be an infinite cardinal, and let $\Dwf\subseteq\Dwf_*$, where $\Dwf_*$ is the class of all pairs $(D,\bar{I})$ such that $D$ is an ultrafilter on some $\pts(K)$ (with $K\neq \emptyset$) and $\bar{I} = \seq{I_k}{ k\in K}$ is a pairwise disjoint family of finite non-empty sets.
        
        The poset $\Por$ is \emph{$\mu$-$\Dwf$-linked}, witnessed by $\seq{Q_{\alpha,\varp}}{\alpha<\mu,\ \varp\in(0,1)\cap \Q}$, if:
        \begin{enumerate}[label = \normalfont (\roman*)]
            \item Each $Q_{\alpha,\varp}$ is $(D,\bar I,\varp)^*$-linked for any $(D,\bar I) \in \Dwf$.
            \item\label{ul2d} For $\varp\in(0,1)\cap \Q$, $\bigcup_{\alpha<\omega} Q_{\alpha,\varp}$ is dense in $\Por$.
        \end{enumerate}

        \item\label{ul3} The poset $\Por$ is \emph{uniformly $\mu$-$\Dwf$-linked} if there is some $\seq{Q_{\alpha,\varp}}{\alpha<\mu,\ \varp\in(0,1)\cap \Q}$ as above, such that in~\ref{ul1} the name $\dot D'$ only depends on $(D,\bar I)$ (and not on any $Q_{\alpha,\varp}$, although we may have different limits on each $Q_{\alpha,\varp}$).
    \end{enumerate}
    We write \emph{$\sigma$-$\Dwf$-linked} when $\mu=\aleph_0$.
\end{definition}

\begin{remark}\label{remuflim}
   In~\ref{ul1} of \autoref{def:uflim}, if $\bar I$ is composed by singletons, say $I_k = \{k\}$, then \autoref{ufext} is equivalent to
   \[\lim \bar p \Vdash \set{k\in K}{p_k\in \dot G}\in \dot D',\]
   which means that $Q$ \emph{has $D$-limits} (cf.~\cite{GMS,Mmini,nuloaditivo}).
   
   In the case that $\Dwf$ is the collection of all pairs $(D,\bar I)$ such that $D$ is an ultrafilter on $\pts(\omega)$ and $\bar I =\seq{\{k\}}{k<\omega}$, we obtain the notion of (uniform) $\mu$-uf-lim-linked as in~\cite{Mmini,nuloaditivo}, which is the notion of forcings with ultrafitler limits from~\cite{GMS}.
\end{remark}

\begin{example}\label{ex:ul}
\ 
\begin{enumerate}[label=\normalfont (\arabic*)]
\item Similar to \autoref{exm:famlk}~\ref{famsing}, all singletons are $(D,\bar I,\varp)^*$-linked for any tuple $(D,\bar I,\varp)$, and $\Por$ is uniformly $|\Por|$-$\Dwf_*$-linked. In particular, Cohen forcing is $\sigma$-$\Dwf_*$-linked.

\item From~\cite{GMS,BCM} we have that several posets associated with localization and anti-localization are uniformly $\sigma$-uf-lim-linked. However, $\Ebb$ and the localization posets are not $\sigma$-$\Dwf_*$-linked because, similar to \autoref{exm:famlk}~\ref{famlkE}, we have that $\sigma$-$\Dwf_*$-linked poset do not increase $\non(\Ewf)$. The case of anti-localization posets is not clear (likewise in the case of fam-limits).

\item Cardona and the author~\cite{nuloaditivo} presented uniformly $\sigma$-uf-lim-linked posets increasing $\non(\mathcal{MA})$ and $\add(\mathcal{SN})$, where $\mathcal{MA}$ denotes the ideal of meager-additive subsets of $2^\omega$, and $\mathcal{SN}$ is the ideal of strong measure zero subsets of $2^\omega$. Since $\non(\mathcal{MA}) \leq \non(\Ewf)$, many instances of the first poset cannot be $\sigma$-$\Dwf_*$-linked neither $\sigma$-$\calY_*$-linked.
\end{enumerate}
In contrast with~\autoref{exm:famlk}~\ref{randomfam}, it is unclear whether random forcing is $\sigma$-uf-lim-linked.
\end{example}

Similar to \autoref{char}, we can characterize uniform $\mu$-$\Dwf$-linkedness as follows.

\begin{theorem}\label{charuf}
 Let $\mu$ be a cardinal, $\Por$ a poset, $\seq{Q_{\alpha,\varp}}{\alpha<\mu,\ \varp\in(0,1)\cap\Q}$ a sequence of subsets of $\Por$, $\Dwf\subseteq\Dwf_*$ and, for each $(D,\bar I)\in \Dwf$, $\alpha<\mu$ and $\varp\in(0,1)\cap\Q$, ${\lim}^{D,\bar I}\colon Q^W_{\alpha,\varp}\to \Por$ where $W:=\bigcup_{k\in K} I_k$. Then, the following statements are equivalent.
\begin{enumerate}[label=\normalfont (\Roman*)]
\item\label{charuf1} $\Por$ is uniformly $\mu$-$\Dwf$-linked witnessed by $\seq{Q_{\alpha,\varp}}{\alpha<\mu,\ \varp\in(0,1)\cap\Q}$.
\item\label{charuf2} \autoref{def:uflim}~\ref{ul2}~\ref{ul2d} holds and, for any
    \begin{multicols}{2} 
    \begin{itemize}
        \item $(D,\bar I)\in\Dwf$,
        \item $i^*<\omega$,
        \item $(\alpha_i,\varp_i)\in \mu\times ((0,1)\cap \Q)$,
        \item $\bar r^i = \seq{r^i_\ell}{\ell\in W} \in Q_{\alpha_i,\varp_i}^W$ for $i<i^*$,
        \item $a\in D$, and
        \item $q\in \Por$ stronger than $\lim^{D,\bar I} \bar r^i$ for all $i<i^*$,
    \end{itemize}
    \end{multicols}    
    there are some $q'\leq q$ and $k\in a$ such that, for all $i<i^*$,
    \begin{equation}\label{-ul2}
       \frac{|\set{\ell\in I_k}{q' \leq r^i_\ell}|}{|I_k|} \geq 1-\varp_i.
	 \end{equation}
\end{enumerate}
\end{theorem}
\begin{proof}
\ref{charuf1}${}\Rightarrow{}$\ref{charuf2}: Assume~\ref{charuf1} and the assumptions of~\ref{charuf2}. Let $\dot D'$ be as in \autoref{def:uflim}~\ref{ul3}, and let $G$ be $\Por$-generic over $V$ such that $q\in G$, and $D':=\dot D'[G]$. In $V[G]$, since $q\leq \lim^{D,\bar I}\bar r^i$ for all $i<i^*$, by \autoref{ufext}
\[\lset{k\in K}{f_i(k) \geq 1-\varp_i}\in D' \text{ where }f_i(k):=\frac{|\set{\ell \in I_k}{r^i_\ell \in \dot G}|}{|I_k|}.\]
Therefore, $a\cap\bigcap_{i<i^*}\set{k\in K}{f_i(k) \geq 1-\varp_i}\in D'$, so this set is non-empty and contains some element $k$.
   
   Back in $V$, find $q'\leq q$ forcing the above and such that either $q'\leq r^i_\ell$ or $q'\perp r^i_\ell$ for all $i<i^*$ and $\ell \in I_k$. Then, $q'$ decides the value of $f_i(k)$ for all $i<i^*$, even more, $q'$ forces
   \[f_i(k) = \frac{\set{\ell\in I_k}{q'\leq r^i_\ell}}{|I_k|}.\]
   Then,~\autoref{-ul2} follows.
   
   \ref{charuf2}${}\Rightarrow{}$\ref{charuf1}: Assume~\ref{charuf2} and $(D,\bar I)\in \Dwf$. Let $G$ be $\Por$-generic over $V$ and work in $V[G]$. Consider the set
   \[\Omega := \lset{(\alpha,\varp,\bar r)}{\alpha<\mu,\ \varp\in(0,1)\cap\Q,\ \bar r\in Q^W_{\alpha,\varp}\cap V,\ {\lim}^{D,\bar I}\bar r \in G}.\]
   
   For each $\alpha<\mu$, $\varp\in(0,1)\cap \Q$ and $\bar r\in Q^W_{\alpha,\varp}\cap V$, define
   \[f_{\alpha,\varp,\bar r}(k):= \frac{\set{\ell \in I_k}{r_\ell\in \dot G}}{|I_k|},\qquad a_{\alpha,\varp,\bar r}:=\set{k\in K}{f_{\alpha,\varp,\bar r}(k)\geq 1-\varp}.\]
   To prove~\ref{charuf1}, it is enough to check that the family $D\cup \lset{a_{\alpha,\varp,\bar r}}{(\alpha,\varp,\bar r)\in\Omega}$ has the finite intersection property. Indeed, assume
   \begin{multicols}{2} 
    \begin{itemize}
        \item $i^*<\omega$,
        \item $(\alpha_i,\varp_i,\bar r^i)\in \Omega$ for $i<i^*$, and
        \item $a\in D$.
    \end{itemize}
    \end{multicols}
    Back in $V$, let $q\in \Por$ be a condition forcing the above such that, wlog, $q\leq \lim^{D,\bar I}\bar r^i$ for all $i<i^*$. Then, by~\ref{charuf2}, there exists some $q'\leq q$ in $\Por$ and some $k\in a$ satisfying \autoref{-ul2} for all $i<i^*$. This density argument allows to find such a $q'$ in $G$. Therefore, in $V[G]$, for any $i<i^*$,
    \[f_{\alpha_i,\varp_i,\bar r^i}(k) \geq \frac{|\set{\ell\in I_k}{q' \leq r^i_\ell}|}{|I_k|} \geq 1-\varp_i,\]
    so $k\in a\cap \bigcap_{i<i^*}a_{\alpha_i,\varp_i,\bar r^i}$.
\end{proof}

The purpose of this section is to prove the following.

\begin{theorem}\label{thm:Etlduf}
  Under \autoref{hyplog}, $\Etld$ is uniformly $\sigma$-$\Dwf_*$-linked.
\end{theorem}

The strategy to prove this theorem is similar to \autoref{mainEtld}.

\begin{theorem}\label{uf-lim1}
   Let $K\neq\emptyset$, $D$ an ultrafilter on $\pts(K)$ and $\bar I = \seq{I_k}{k\in K}$ a partition of a set $W$ into finite sets. Then,   
    for any $s\in T_*\menos\{\la\ \ra\}$ there is a function $\lim^{D,\bar I}\colon (E'_{s})^W\to \Etld$ such that $\lim^{D,\bar I}\bar r$ has stem $s$ for any $\bar r \in (E'_{s})^W$, and satisfying: For any
    \begin{multicols}{2}
    \begin{itemize}
		\item $i<i^*$
		\item $s_i\in T_*\menos\{\la\ \ra\}$,
		\item $\bar r^i = \seq{r^i_\ell}{\ell\in W} \in (E'_{s_i})^W$ for $i<i^*$,
        \item $a\in D$, and
        \item $q\in \Etld$ stronger than $\lim^{D,\bar I} \bar r^i$ for all $i<i^*$,
	 \end{itemize}
	 \end{multicols}
	 there are $q'\leq q$ in $\Etld$ and $k\in a$ such that, for any $i<i^*$,
	  \begin{equation}\label{lprop}
       \frac{|\set{\ell\in I_k}{q' \leq r^i_\ell}|}{|I_k|}  > 1-\frac{1}{|s_i|}.
 	 \end{equation}
\end{theorem}
\begin{proof}
We proceed as in the proof of \autoref{fam-lim1}. Let $s\in T_*$. We first show how to define $\lim^{D,\bar I}\bar r$ for $\bar r = \seq{r_\ell}{\ell\in W}\subseteq (E'_{s})^W$. For each $k\in K$, let $r^*_k$ be the pseudo-fusion of $\seq{r_\ell}{\ell\in I_k}$ as in \autoref{psf}. For each $t\in T_*$ set $Z_t= Z^{\bar r}_t:= \set{k\in K}{t\in r^*_k}$. Define $r^* :=\lim^{D,\bar I}\bar r$ such that, for $t\in T_*$,
\[t\in r^* \text{ iff }Z_t\in D.\]
Since the tree $T_*$ is finitely branching, we obtain that, for any $t\in r^*$, $\set{k\in K}{\suc_{r^*}(t)= \suc_{r^*_k}(t)}\in D$. This implies that $r^*$ is a tree with stem $s$ and 
\[\|r^*\|_t\geq 1 + \frac{1}{|s|} + 2 \log_{b_*(t)}|t|\]
for any $t\in r^*$ above $s$. Hence $r^*\in \Etld$.
    
    We now prove that this limit works. Work under the assumptions of the bullet points. For $i<i^*$ and $k\in K$, let $r^i_k$ be the pseudo-fusion of $\seq{r^i_\ell}{\ell\in I_k}$, $r^i:=\lim^{D,\bar I}\bar r^i$ and $Z^i_t:= Z^{\bar r^i}_t$. So we have that $q\leq r^i$ for all $i<i^*$.
    
    Pick $t\in q$ large enough such that $q|^t$ and $r^i|^t$ are in $E_{t,i^*+1}$ for $i<i^*$ (which is fine because $t\in r^i$ for all $i<i^*$). Then $a\cap\bigcap_{i<i^*}Z^i_t \in D$, so this intersection is non-empty and we can pick some $k$ in there, i.e.\ $k\in a$ and $t\in r^i_k$ for all $i<i^*$. Hence, by \autoref{klinked}, there is a lower bound $q_1$ of $\{q\}\cup \set{r^i_k}{i<i^*}$. Apply \autoref{psf}~\ref{bul2} $i^*$-many times to find $q'\leq q_1$ such that, for $i<i^*$,
    \[
       \frac{|\set{\ell\in I_k}{q' \leq r^i_\ell}|}{|I_k|}  > 1-\frac{1}{|s_i|}.\qedhere
 	 \]
\end{proof}

\begin{proof}[Proof of \autoref{thm:Etlduf}]
For $s\in T_*\menos\{\la\ \ra\}$ and $\varp \in (0,1)\cap \Q$, define
\[Q'_{s,\varp}:=
\left\{
\begin{array}{ll}
E'_s & \text{if $\displaystyle \frac{1}{|s|}\leq \varp$,}\\[3ex]
\{T_*|^s\} & \text{otherwise.}
\end{array}
\right.\]
It is easy to show that, for any $\varp \in (0,1)\cap \Q$,
\[\bigcup_{s\in T_*\menos\{\la\ \ra\}}Q'_{s,\varp} \supseteq \bigcup\lset{E'_s}{s\in  T_*\menos\{\la\ \ra\},\ |s|\geq \frac{1}{\varp}} \text{ is dense in }\Etld.\]

By checking~\ref{charuf2} of \autoref{charuf}, 
we show that $\seq{Q'_{s,\varp}}{s\in T_*\menos\{\la\ \ra\},\ \varp\in(0,1)\cap \Q}$ witnesses that $\Etld$ is $\sigma$-$\Dwf_*$-linked. Assume
\begin{multicols}{2} 
    \begin{itemize}
        \item $(D,\bar I)\in\Dwf_*$,
        \item $i^*<\omega$,
        \item $(s_i,\varp_i)\in (T_*\menos\{\la\ \ra\})\times ((0,1)\cap \Q)$,
        \item $\bar r^i = \seq{r^i_\ell}{\ell\in W} \in Q'^W_{s_i,\varp_i}$ for $i<i^*$,
        \item $a\in D$, and
        \item $q\in \Etld$ stronger than $\lim^{D,\bar I} \bar r^i$ for all $i<i^*$.
    \end{itemize}
    \end{multicols}   
    When $\varp_i<\frac{1}{|s_i|}$ we are dealing with the singleton $Q'_{s_i,\varp_i}=\{T_*|^{s_i}\}$, for which the sequence $\bar r^i$ is constant and $\lim^{D,\bar I}\bar r^i = T_*|^{s_i}$, so $q \leq T_*|^{s_i}$. Now, apply \autoref{uf-lim1} to those $i<i^*$ such that $\frac{1}{|s_i|}\leq \varp_i$, and find $q'\leq q$ in $\Etld$ and $k\in a$ satisfying \autoref{lprop} (for those $i$). Then, whenever $\frac{1}{|s_i|}\leq \varp_i$,
    \[\frac{|\set{\ell\in I_k}{q' \leq r^i_\ell}|}{|I_k|}  > 1-\frac{1}{|s_i|}\geq 1-\varp_i.\]
    On the other hand, whenever $\varp_i<\frac{1}{|s_i|}$,
    \[\frac{|\set{\ell\in I_k}{q' \leq r^i_\ell}|}{|I_k|} = 1 > 1-\varp_i.\qedhere\]
\end{proof}

\subsection*{Acknowledgments}
This note is developed for the proceedings of the RIMS Set Theory Workshop 2023 \emph{Large Cardinals and the Continuum}, held at Kyoto University RIMS. The author thanks the organizer, Professor Hiroshi Fujita from Ehime University, for letting him participate with a talk at the Workshop and submit a paper to this proceedings.

This work is supported by the Grants-in-Aid for Scientific Research (C) 23K03198, Japan Society for the Promotion of Science

{\footnotesize
\bibliography{left}

\begin{thebibliography}{CMPLUZ}

\bibitem[BCM21]{BCM}
J\"{o}rg Brendle, Miguel~A. Cardona, and Diego~A. Mej\'{\i}a.
\newblock Filter-linkedness and its effect on preservation of cardinal
  characteristics.
\newblock {\em Ann. Pure Appl. Logic}, 172(1):102856, 2021.

\bibitem[Bre91]{Br}
J{\"o}rg Brendle.
\newblock Larger cardinals in {C}icho\'n's diagram.
\newblock {\em J. Symbolic Logic}, 56(3):795--810, 1991.

\bibitem[Car23]{CardonaRIMS}
Miguel~A. Cardona.
\newblock Controlling the uniformity of the ideal generated by the {$F_\sigma$}
  measure zero subsets of the reals.
\newblock Talk at the RIMS Set Theory Workshop \emph{Large Cardinals and the
  Continuum}, Kyoto University,
  \href{https://tenasaku.com/RIMS2023/slides/cardona-rims2023.pdf}{https://tenasaku.com/RIMS2023/slides/cardona-rims2023.pdf},
  2023.

\bibitem[CM23]{CMlocalc}
Miguel~A. Cardona and Diego~A. Mej\'{\i}a.
\newblock Localization and anti-localization cardinals.
\newblock {\em Ky\={o}to Daigaku S\=urikaiseki Kenky\=usho K\=oky\=uroku},
  2261:47--77, 2023.
\newblock \href{https://arxiv.org/abs/2305.03248}{arXiv:2305.03248}.

\bibitem[CM24]{nuloaditivo}
Miguel~A. Cardona and Diego~A. Mej\'ia.
\newblock Uniformity numbers of the null-additive and meager-additive ideals.
\newblock Preprint, \href{https://arxiv.org/abs/2401.15364}{arXiv:2401.15364},
  2024.

\bibitem[CMPU]{CMUP}
Miguel~A. Cardona, Diego~A. Mejía, Carlos~M. Parra-Londoño, and Andrés~F.
  Uribe-Zapata.
\newblock Finitely additive measures on {B}oolean algebras.
\newblock In preparation.

\bibitem[CMU]{CMU}
Miguel~A. Cardona, Diego~A. Mejía, and Andrés~F. Uribe-Zapata.
\newblock A general theory of iterated forcing using finitely additive
  measures.
\newblock In preparation.

\bibitem[GKMS22]{GKMS}
Martin Goldstern, Jakob Kellner, Diego~A. Mej\'{\i}a, and Saharon Shelah.
\newblock Cicho\'{n}'s maximum without large cardinals.
\newblock {\em J. Eur. Math. Soc. (JEMS)}, 24(11):3951--3967, 2022.

\bibitem[GKS19]{GKS}
Martin Goldstern, Jakob Kellner, and Saharon Shelah.
\newblock Cicho\'{n}'s maximum.
\newblock {\em Ann. of Math. (2)}, 190(1):113--143, 2019.

\bibitem[GMS16]{GMS}
Martin Goldstern, Diego~Alejandro Mej{\'{\i}}a, and Saharon Shelah.
\newblock The left side of {C}icho\'n's diagram.
\newblock {\em Proc. Amer. Math. Soc.}, 144(9):4025--4042, 2016.

\bibitem[HS16]{HSh}
Haim Horowitz and Saharon Shelah.
\newblock Saccharinity with ccc.
\newblock Preprint, \href{https://arxiv.org/abs/1610.02706}{arXiv:1610.02706},
  2016.

\bibitem[Kel59]{Kelley}
J.~L. Kelley.
\newblock Measures on {B}oolean algebras.
\newblock {\em Pacific J. Math.}, 9:1165--1177, 1959.

\bibitem[KST19]{KST}
Jakob Kellner, Saharon Shelah, and Anda~R. T\u{a}nasie.
\newblock Another ordering of the ten cardinal characteristics in
  {C}icho\'{n}'s diagram.
\newblock {\em Comment. Math. Univ. Carolin.}, 60(1):61--95, 2019.

\bibitem[Mej24]{Mmini}
Diego~A. Mejía.
\newblock Forcing techniques for {C}icho\'n's {M}aximum.
\newblock Lecture notes for the mini-course of the same name at the University
  of Vienna, 2023/24.

\bibitem[MU23]{MUrandom}
Diego~A. Mej\'ia and Andres Uribe-Zapata.
\newblock The measure algebra adding $\theta$-many random reals is
  $\theta$-fam-linked.
\newblock Preprint, \href{https://arxiv.org/abs/2312.13443}{arXiv:2312.13443},
  2023.

\bibitem[She00]{ShCov}
Saharon Shelah.
\newblock Covering of the null ideal may have countable cofinality.
\newblock {\em Fund. Math.}, 166(1-2):109--136, 2000.

\bibitem[Uri23]{Uribethesis}
Andrés Uribe-Zapata.
\newblock Iterated forcing with finitely additive measures: applications of
  probability to forcing theory.
\newblock Master's thesis, Universidad Nacional de Colombia, sede Medell\'in,
  2023.
\newblock
  \href{https://www.researchgate.net/publication/369113160_Iterated_forcing_with_finitely_additive_measures_applications_of_probability_to_forcing_theory}{https://shorturl.at/sHY59}.

\bibitem[Uri24]{Uribelinked}
Andrés~F. Uribe-Zapata.
\newblock The intersection number for forcing notions.
\newblock In this volume, 2024.
\newblock \href{https://arxiv.org/abs/2401.14552}{arXiv:2401.14552}.

\end{thebibliography}
\bibliographystyle{alpha}
}

\end{document}